\documentclass[reqno]{amsart}
\usepackage{epsf}
\usepackage{amsfonts}
\usepackage{amsmath}
\usepackage{amsthm}
\usepackage[utf8]{inputenc}
\usepackage[T1]{fontenc}
\usepackage{graphicx}  
\usepackage{multicol}
\usepackage{caption}
\usepackage{color}
\usepackage{subcaption}
\usepackage{hyperref}
\usepackage{mathtools}
\usepackage{thmtools}
\usepackage{thm-restate}
\usepackage{cleveref}
\newtagform{brackets}{[}{]}
\usetagform{brackets}
\usepackage{soul}
\usepackage{float}
\allowdisplaybreaks
\usepackage{etoolbox}

\theoremstyle{plain}
\newtheorem{Thm}{Theorem}
\newtheorem{Lma}{Lemma}
\newtheorem{Prop}{Proposition}
\newtheorem{Cor}{Corollary}

\theoremstyle{definition}
\newtheorem{Def}{Definition}
\newtheorem*{Not}{Notation}

\newtheorem{Ex}{Example}
\newcommand{\tth}{{\text{th}}}
\newcommand{\nth}{n^\tth}
\newcommand{\rth}{r^\tth}

\newcommand{\kth}{k^\tth}

\newcommand{\propref}[1]{Proposition~{\bf\ref{#1}}}

\newcommand{\threepropref}[3]{Propositions~{\bf\ref{#1}},~{\bf\ref{#2}}~and~{\bf\ref{#3}}}
\newcommand{\lemref}[1]{Lemma~{\bf\ref{#1}}}
\newcommand{\figref}[1]{Figure~{\bf\ref{#1}}}
\newcommand{\Eqref}[1]{Equation~{\bf\ref{#1}}}
\newcommand{\secref}[1]{Section~{\bf\ref{#1}}}

\newcommand{\twoeqref}[2]{Equations~{\bf\ref{#1}}~and~{\bf\ref{#2}}}

\newcommand{\conseceqref}[2]{Equations~{\bf\ref{#1}}~-~{\bf\ref{#2}}}
\newcommand{\exref}[1]{Example~{\bf\ref{#1}}}

\newcommand{\bbZ}{\mathbb{Z}}
\newcommand{\bbN}{\mathbb{N}}
\newcommand{\dsum}{\displaystyle\sum}
\newcommand{\grt}{T(c,d,d_1,d_2)}

\everymath{\displaystyle}

\begin{document}
\title{Generalized Rascal Triangles}
\date{\today}
\author{Philip~K. Hotchkiss}
\address{Mathematics Department\\
Westfield State University\\
Westfield, MA 01086}
\email[Philip~K. Hotchkiss]{photchkiss@westfield.ma.edu}
\maketitle

\begin{abstract}
	The Rascal Triangle was introduced by three middle school students in 2010, and in this paper we describe number triangles that are generalizations of the Rascal Triangle and show that these Generalized Rascal Triangles are characterized by arithmetic sequences on all diagonals as well as Rascal-like multiplication and addition rules. 
\end{abstract}

\section{Introduction}\label{Sect:Intro}

In 2010, middle school students Alif Anggaro, Eddy Liu and Angus Tulloch \cite{ALT} were asked to determine the next row of numbers in the following triangular array:
	\begin{figure}[H]
		\centering
		\includegraphics[scale=0.5]{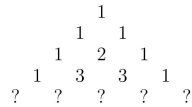}
		\caption{A triangular array.}\label{Fi:RasTriStart}
	\end{figure}
Instead of providing the expected answer
	\begin{center}
		1\qquad 4\qquad 6\qquad 4\qquad 1
	\end{center}
 from Pascal's Triangle; they produced the row
	\begin{center}
		1\qquad 4\qquad 5\qquad 4\qquad 1.
	\end{center}
They did this by using the rule that the outside numbers are 1s and the inside numbers are determined by the {\bf diamond formula}
	\begin{equation*}\label{Eq:ALT}
		\mathbf{ South} = \dfrac{\mathbf{East}\cdot \mathbf{West} + 1}{\mathbf{North}}
	\end{equation*}
where {\bf North}, {\bf South}, {\bf East} and {\bf West} form a diamond in the triangular array as in \figref{Fi:nsew}.
	\begin{figure}[H]
		\centering
		\includegraphics[scale=0.5]{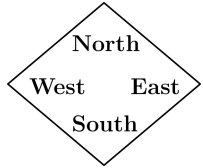}
		\caption{{\bf North}, {\bf South}, {\bf East} and {\bf West} entries in a triangular array.}\label{Fi:nsew}
	\end{figure}
	
Continuing with this rule Anggaro, Liu and Tulloch created a number triangle they called the {\bf Rascal Triangle}.
	\begin{figure}[H]
		\centering
		\includegraphics[scale=0.5]{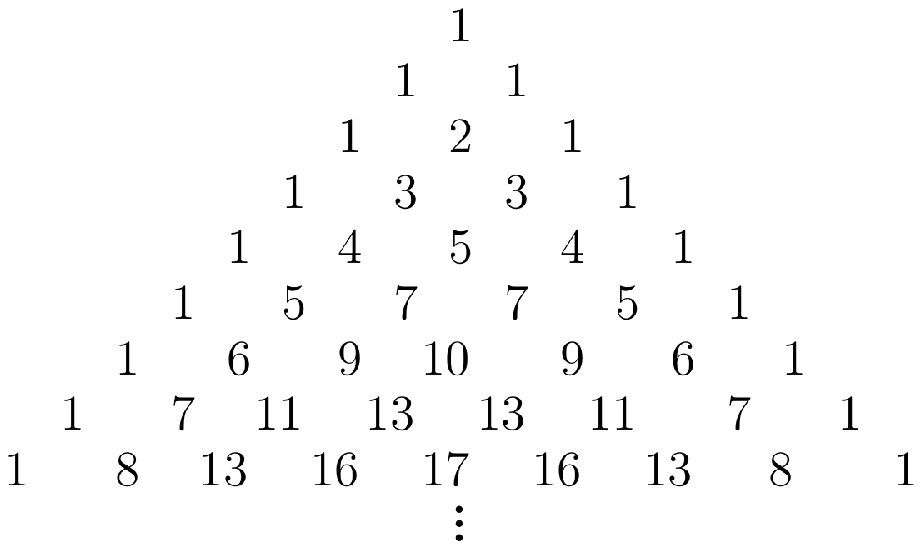}
		\caption{The Rascal Triangle.}\label{Fi:RasTri}
	\end{figure}

Because the diamond formula involves division, their instructor challenged Anggaro, Liu and Tulloch to prove that it would always result in an integer.  They did this by observing that the diagonals in the Rascal Triangle formed arithmetic sequences.  In particular, the $\kth$ entry on the $\rth$ diagonal running from right to left is given by $1+rk$, where $r=0$ corresponds to the outside diagonal consisting of all 1s, and $k=0$ corresponds to the first entry on each diagonal.  Thus, if {\bf North} = $1+(r-1)(k-1)$, then {\bf East} = $1+r(k-1)$, {\bf West} = $1+(r-1)k$ and {\bf South} = $1+rk$, 
\begin{figure}[H]
	\centering
	\includegraphics[scale=0.5]{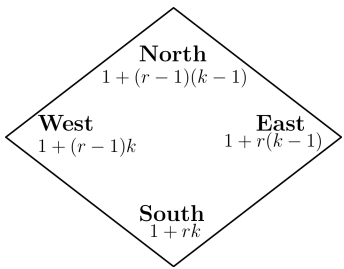}
	\caption{Algebraic Representation of \textbf{North}, \textbf{East}, \textbf{West} and \textbf{South}.}\label{Fi:NSEWReps}  
\end{figure}
and a straight forward calculation verifies that
	\begin{equation*}
		\mathbf{South} = \dfrac{\mathbf{East}\cdot \mathbf{West} + 1}{\mathbf{North}}.
	\end{equation*}

In the Spring 2015 semester, a Mathematics for Liberal Arts class taught by my colleague, Julian Fleron, discovered that the Rascal Triangle can also be generated by the rule that the outside numbers are 1s and the inside numbers are determined by the formula
	\begin{equation*}\label{Eq:JFclass}
		\mathbf{South} = \mathbf{East}+ \mathbf{West} -\mathbf{North} + 1.
	\end{equation*}
	
This formula also follows from the arithmetic progressions along the diagonals \cite{JF}.  Thus, the Rascal Triangle has the property that for any diamond containing 4 entries, the {\bf South} entry satisfies two equations: 
	\begin{align}
		\mathbf{South} &= \dfrac{\mathbf{East}\cdot \mathbf{West} + 1}{\mathbf{North}}\label{eq:RasForm1}\\
		\mathbf{South} &= \mathbf{East}+ \mathbf{West} -\mathbf{North} + 1\label{eq:RasForm2}
	\end{align}
	
The fact that both \twoeqref{eq:RasForm1}{eq:RasForm2} can be used to generate the Rascal Triangle was intriguing to me and I assumed that the Rascal Triangle was uniquely defined by either one of the two equations; so I began trying to prove that \Eqref{eq:RasForm2} implied \Eqref{eq:RasForm1} or vice versa.  In addition, I followed Julian Fleron's lead and had some of my mathematics for liberal arts classes look for patterns in the Rascal Triangle, and to my delight they also made some original discoveries \cite{H}.  During the following summer, while exploring the patterns found by my students, I realized that there were other number triangles for which one equation held for the interior entries but the other did not, as will be shown below.   This led me to the realization that there was a larger class of number triangles, of which the Rascal Triangle was but one example.

\section{Motivating Examples and Generalized Rascal Triangles}\label{S:MExGRTs}

In what follows we will use the following conventions.

\begin{Def}
	For any number triangle, the diagonals running from right to left will be called the {\bf major} diagonals while the diagonals running from left to right will be called the {\bf minor} diagonals.
\end{Def}

	\begin{figure}[H]
		\centering
		\begin{subfigure}[b]{0.2\textwidth}
			\includegraphics[width=0.675\textwidth]{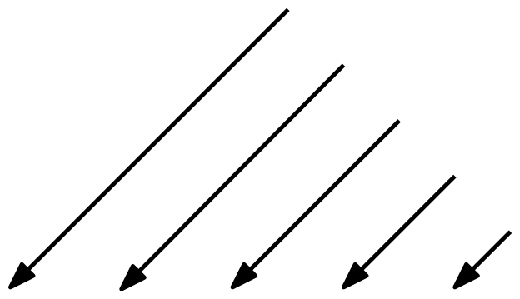}
			\subcaption{major diagonals.}\label{Fi:majord}
		\end{subfigure}
		\hskip0.5in
		\begin{subfigure}[b]{0.2\textwidth}
			\includegraphics[width=0.675\textwidth]{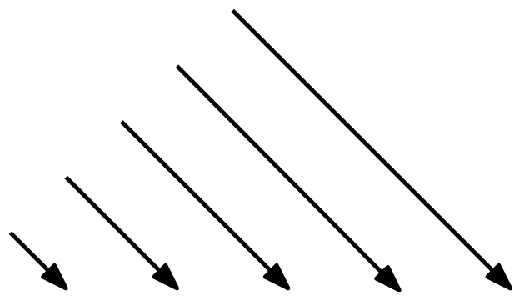}
			\subcaption{Minor diagonals.}\label{Fi:minord}
		\end{subfigure}
	\end{figure}

\begin{Ex}\label{ex1} 
	Let $T$ be the number triangle in \figref{Fi:GenRasTri3};
		\begin{figure}[H]
			\centering
			\includegraphics[scale=0.5]{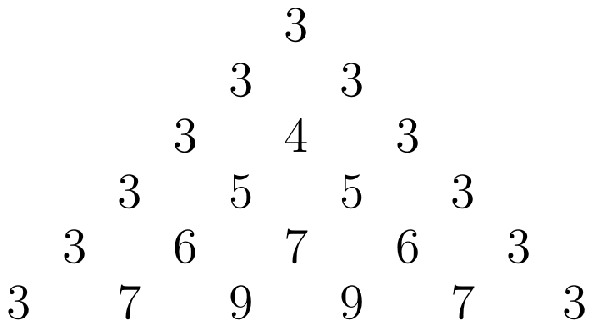}
			\caption{The first six rows of $T$.}\label{Fi:GenRasTri3}
		\end{figure}
	\noindent the interior numbers satisfy \Eqref{eq:RasForm2} but not \Eqref{eq:RasForm1}.  However, the interior numbers do satisfy a modified version of \Eqref{eq:RasForm1}:
	\begin{equation*}
		\mathbf{South} = \dfrac{\mathbf{West}\cdot \mathbf{East} + 3}{\mathbf{North}}
	\end{equation*}
\end{Ex}

\begin{Ex}\label{ex2} 
	Let $S$ be the number triangle in \figref{Fi:GenRasTri5};
		\begin{figure}[H]
			\centering
			\includegraphics[scale=0.5]{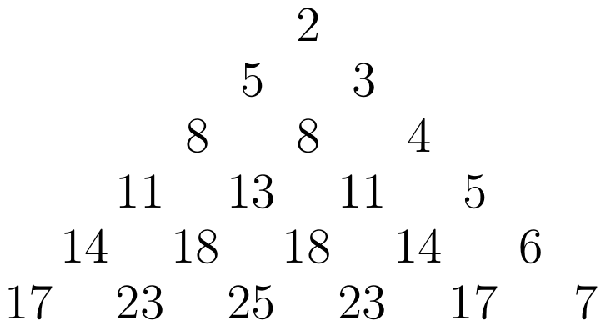}
			\caption{The first six rows of $S$.}\label{Fi:GenRasTri5}
		\end{figure}
	\noindent the interior numbers satisfy \Eqref{eq:RasForm1} but not \Eqref{eq:RasForm2}.  However, the interior numbers do satisfy a modified version of \Eqref{eq:RasForm2}:
	\begin{equation*}
		\mathbf{South} = \mathbf{West} + \mathbf{East} + 2 - \mathbf{North}
	\end{equation*}
\end{Ex}

\begin{Ex}\label{ex3} 
	Let $U$ be the number triangle in \figref{Fi:NumTri1}
		\begin{figure}[H]
			\centering
			\includegraphics[scale=0.5]{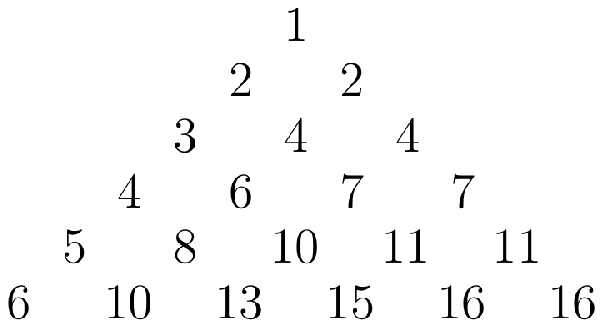}
			\caption{The first six rows of $U$.}\label{Fi:NumTri1}
		\end{figure}
	\noindent the interior numbers satisfy \Eqref{eq:RasForm2}, but there is no modification of \Eqref{eq:RasForm1} that works for all interior numbers.  To see this, consider the diamonds
	\begin{figure}[H]
		\centering
		\includegraphics[scale=0.5]{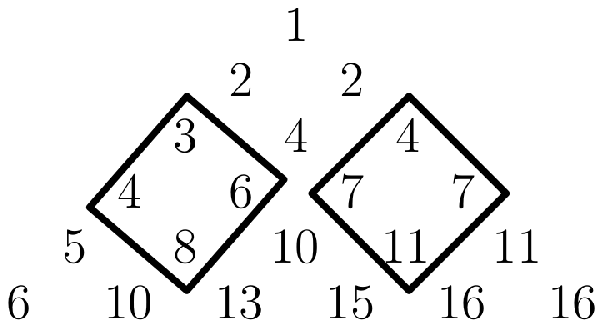}
		\caption{Different diamonds.}\label{Fi:NumTri1Diamonds}
	\end{figure}
	For the diamond on the left, the modification of \Eqref{eq:RasForm1} would need to be
		\begin{align*}
			\mathbf{South} &=\dfrac{\mathbf{West}\cdot \mathbf{East} + 0}{\mathbf{North}}\\
			\intertext{while for the diamond on the right, the modification of \Eqref{eq:RasForm1} would need to be}
			\mathbf{South} &=\dfrac{\mathbf{West}\cdot \mathbf{East} -5}{\mathbf{North}}.
		\end{align*}
\end{Ex}	

\begin{Ex}\label{ex4} 
	Let $V$ be the number triangle whose first six rows are shown in \figref{Fi:NumTri2}.
		\begin{figure}[H]
			\centering
			\includegraphics[scale=0.3]{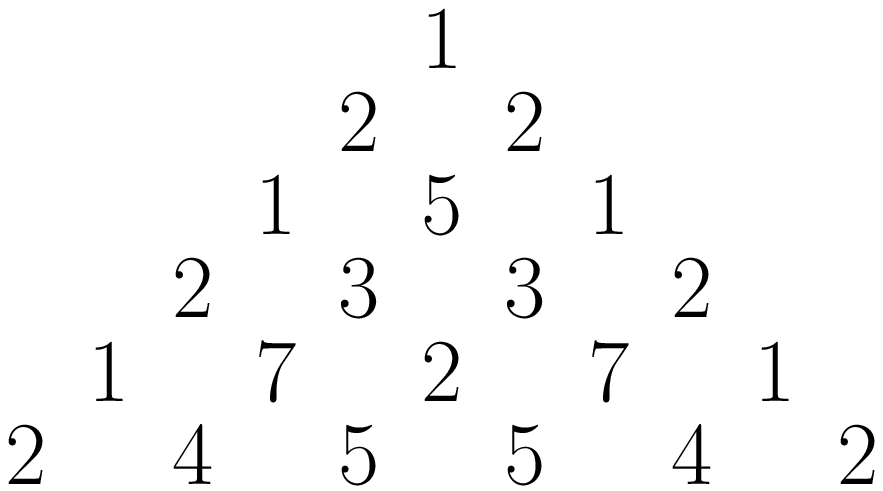}
			\caption{The first six rows of $V$.}\label{Fi:NumTri2}
		\end{figure}
	\noindent The interior numbers satisfy \Eqref{eq:RasForm1}, but there is no modification of \Eqref{eq:RasForm2} that works for all interior numbers.  To see this, consider the diamonds in \figref{Fi:NumTri2Diamonds}.
	\begin{figure}[H]
		\centering
		\includegraphics[scale=0.3]{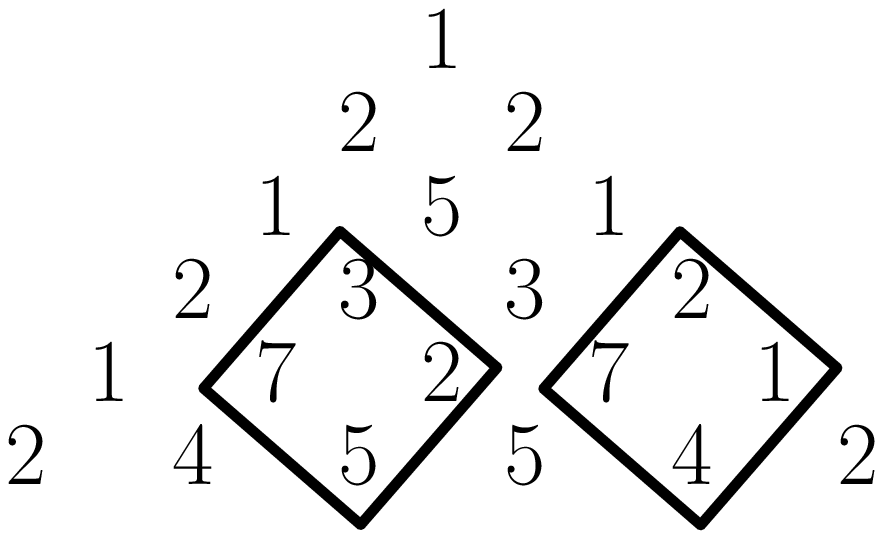}
		\caption{Different diamonds.}\label{Fi:NumTri2Diamonds}
	\end{figure}
	For the diamond on the left, the modification of \Eqref{eq:RasForm2} would need to be
		\begin{align*}
			\mathbf{South} &=\mathbf{West}+ \mathbf{East} -1 - \mathbf{North}\\
			\intertext{while for the diamond on the right, the modification of \Eqref{eq:RasForm2} would need to be}
			\mathbf{South} &=\mathbf{West}+ \mathbf{East} -2 - \mathbf{North}.
		\end{align*}
\end{Ex}	

\begin{Ex}\label{ex5} 
	Let $W$ be the number triangle whose first six rows are shown in \figref{Fi:GenRasTri1}.
		\begin{figure}[H]
			\centering
			\includegraphics[scale=0.5]{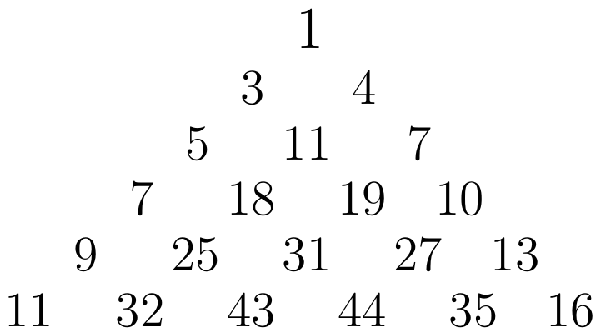}
			\caption{The first six rows of $W$.}\label{Fi:GenRasTri1}
		\end{figure}
	\noindent the interior numbers satisfy neither \twoeqref{eq:RasForm1}{eq:RasForm2}, but there are modifications of both equations that work for all interior numbers,
		\begin{align*}
			\mathbf{South} &= \mathbf{West} + \mathbf{East} + 5 - \mathbf{North}
			\intertext{and}
			\mathbf{South} &=\dfrac{\mathbf{West}\cdot \mathbf{East} -1}{\mathbf{North}}.
		\end{align*}
\end{Ex}

Note that for $T$, $S$ and $W$ all the major and minor diagonals are arithmetic sequences and the interior numbers satisfy equations similar to \twoeqref{eq:RasForm1}{eq:RasForm2}.  Whereas in $U$, some of the diagonals are not arithmetic sequences and although the interior entries in $U$ satisfy \Eqref{eq:RasForm2}, there is no modification of \Eqref{eq:RasForm1} that will work for all of the interior entries.  While in $V$, none of the diagonals are arithmetic sequences and although the interior entries in $V$ satisfy \Eqref{eq:RasForm1} there is no modification of \Eqref{eq:RasForm2} that will work for all of the interior entries.    This suggests that for number triangles with arithmetic sequences on both the major and minor diagonals, the interior numbers satisfy two equations of the form \twoeqref{eq:RasForm1}{eq:RasForm2}.  

This motivates the definitions below. 

\begin{Not}
For a number triangle, $T$, we will use $T_{r,k}$ to denote the $\kth$ entry on the $\rth$ major diagonal with $r=0$ corresponding to the outside major diagonal and $k=0$ corresponding to the first entry on each major diagonal.  With this notation, $T_{0,0}$ corresponds to the top number of $T$.  Note that on the minor diagonals, $T_{r,k}$ denotes the $\rth$ entry on the $\kth$ minor diagonal with $k=0$ corresponding to the outside minor diagonal on the right and $r=0$ corresponding to the first entry on each minor diagonal.   
\end{Not}

\begin{Def}
Let $c,d,d_1,d_2 \in \bbZ$. A number triangle, $T$ is called a \emph{Generalized Rascal Triangle} if  
	\begin{equation}
		T_{r,k} = c + kd_1 + rd_2 + rkd\label{eq:GenRasEq}
	\end{equation} 
for all $r,k\ge 0$.  We will write $\grt$ for the Generalized Rascal Triangle determined by the constants $c,d,d_1,d_2$.
\end{Def}

In a Generalized Rascal Triangle, $\grt$, the constant $c$ is the top entry, $d_1$ is the arithmetic difference along the outside major diagonal, $d_2$ is the arithmetic difference along the outside minor diagonal and $d$ is the change in the arithmetic differences as we move from one major diagonal to the next or move from one minor diagonal to the next.  In particular, the $\rth$ major diagonal is the arithmetic sequence
	\begin{align}
		M_r(k) &= (c+rd_2) + k(d_1+rd);\label{eq:majorarthseq}\\
		\intertext{and the $\kth$ minor diagonal is the arithmetic sequence}
		m_k(r) &= (c+kd_1) + r(d_2+kd).\label{eq:minorarthseq}
	\end{align}

\begin{Def}
	Let $d, D \in \bbZ$ and let $T$ be a number triangle.  If the interior numbers satisfy the equation
	\begin{align}
		T_{r,k} &= \dfrac{T_{r-1,k}\cdot T_{r,k-1} + D}{T_{r-1,k-1}}\label{eq:RascalMult}\\
		\intertext{we call this a \emph{Rascal-like multiplication rule} with multiplicative constant, $D$; and if the interior numbers satisfy}
		T_{r,k} &= T_{r-1,k} + T_{r,k-1} + d - T_{r-1,k-1}\label{eq:RascalAdd}
	\end{align}
	that will be called a \emph{Rascal-like addition rule} with additive constant $d$.
\end{Def}

\begin{Ex}
	The Generalized Rascal Triangle $T(1,1,0,0)$, which is defined by the equation
		$$T_{r,k} = 1 + 0k + 0r + rk = 1+rk$$
	has a constant sequence of 1s on the outside diagonals, and the arithmetic differences increase by 1 as we move from one major (resp., minor) diagonal to the next. This is, of course, the Rascal Triangle.
	\begin{figure}[H]
		\centering
		\includegraphics[scale=0.5]{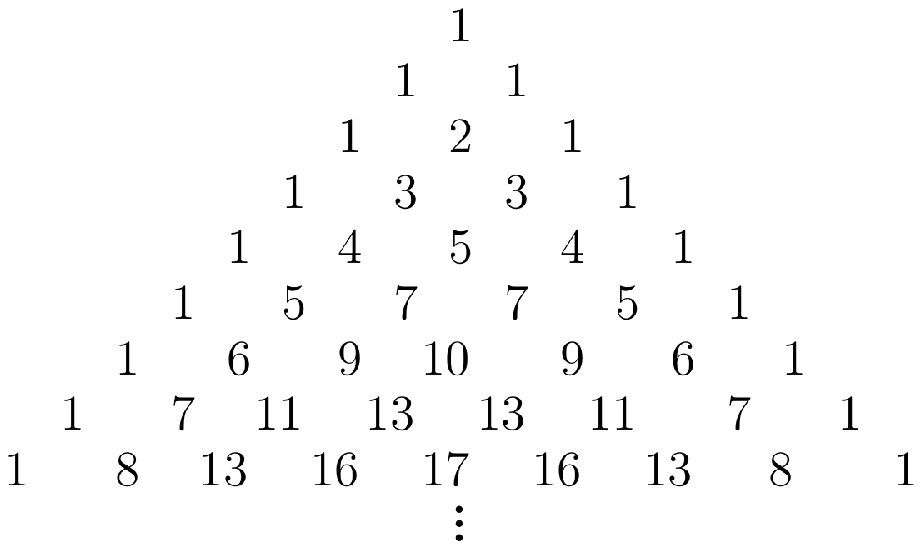}
		\caption{The Rascal Triangle.}\label{Fi:RasTri2}
	\end{figure}
	We will denote this triangle by $R$, and it's equation by $R_{r,k}=1+rk$.  As was mentioned earlier, the interior numbers satisfy both the addition rule
		\begin{align*}
			R_{r,k} &= R_{r-1,k} + R_{r, k-1} + 1 - R_{r-1,k-1}.\\
			\intertext{and the multiplication rule}
			R_{r,k} &= \dfrac{R_{r-1,k}\cdot R_{r,k-1} +1}{R_{r-1,k-1}}.
		\end{align*}
\end{Ex}

\begin{Ex}
	The Generalized Rascal Triangle, $T(2,2,3,1)$, is the number triangle determined by the equation 
		$$T_{r,k} = 2 + 2k + 3r + rk.$$ 
	This number triangle has a top entry of 2, the outside major diagonal has an arithmetic difference of 3, and the outside minor diagonal has an arithmetic difference of 1. Moreover, the arithmetic differences change by 2 as we move from one major (resp., minor) diagonal to the next. This results in the number triangle, $S$, from \exref{ex2}.  
		\begin{figure}[H]
			\centering
			\includegraphics[scale=0.5]{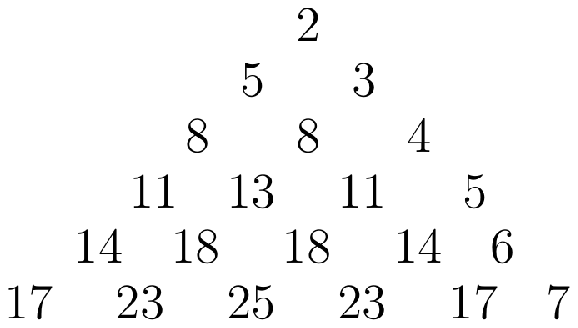}
			\caption{The first six rows of $S=T(2,1,3,2)$.}\label{Fi:GenRasTri6}
		\end{figure}
	
	The Rascal-like addition rule for $S$ is 
		\begin{align*}
			T_{r,k} &= T_{r-1,k} + T_{r, k-1} + 2 - T_{r-1,k-1}.\\
			\intertext{and the Rascal-like multiplication rule for $S$ is}
			T_{r,k} &= \dfrac{T_{r-1,k}\cdot T_{r,k-1} +1}{T_{r-1,k-1}}
		\end{align*}
\end{Ex}

\section{Addition and Multiplication Rules for Generalized Rascal Triangles}\label{S:GRTProps}

It is easy to show that every Generalized Rascal Triangle has a Rascal-like addition rule, \Eqref{eq:RascalAdd}, and a Rascal-like multiplication rule, \Eqref{eq:RascalMult} (see \propref{prop:GRTAdd} below). Our initial (naive) assumption was that if a number triangle $T$ had a Rascal-like addition or Rascal-like multiplication rule, it was a Generalized Rascal Triangle. However, as illustrated above, the existence of just a Rascal-like addition rule or just a Rascal-like multiplication rule for the interior numbers is not sufficient for a number triangle to be a Generalized Rascal Triangle.    Nevertheless, as we will show below, if the outside major and minor diagonals are arithmetic sequences, the existence of either a Rascal-like addition rule or a Rascal-like multiplication rule for the interior numbers is both necessary and sufficient for a number triangle to be a Generalized Rascal Triangle.

We begin with two lemmas whose proofs are left to the reader.

\begin{Lma}\label{L:constantr+k}
	Let $T$ be a number triangle. If $T_{r,k}$ is on the $\nth$ row,  then $r+k=n$. 
\end{Lma}
	\begin{figure}[H]
		\centering
		\includegraphics[scale=0.5]{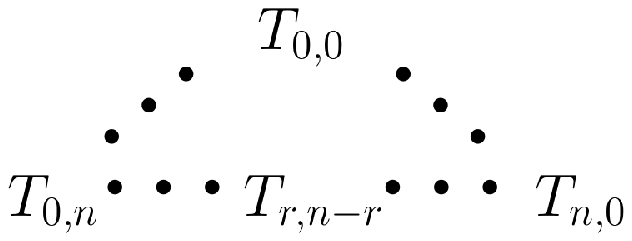}
		\caption{Row $n$.}\label{Fi:rown}
	\end{figure}

\begin{Lma}\label{(linE)(linW)=(linS)(linN)}
	Let $c, d_1, d_2, r, k, d \in \bbZ$, then 
		\begin{align*}
			\Big(c &+ (k-1)d_1+ rd_2  + r(k-1)d\Big)\Big(c  + kd_1+ (r-1)d_2 + (r-1)kd\Big) +cd-d_1d_2\\
			&= \Big(c+kd_1+rd_2+rkd\Big)\Big(c + (k-1)d_1 + (r-1)d_2 + (r-1)(k-1)d\Big).
		\end{align*}
\end{Lma}
\medskip

\subsection{Generalized Rascal Triangles satisfy Rascal-like addition and Rascal-like multiplication rules}\label{Ss:GRTAddMult} We first show that every Generalized Rascal Triangle satisfies Rascal-like addition and Rascal-like multiplication rules.
 
\begin{Prop}\label{prop:GRTAdd}
	Let $c, d, d_1, d_2 \in \bbZ$ and $T=\grt$ be the associated Generalized Rascal Triangle; then for $r,k\ge 1$
		\begin{align*}
			T_{r,k} &= T_{r,k-1} + T_{r-1,k} + d - T_{r-1,k-1}\\
			\intertext{and whenever $T_{r-1,k-1}\ne 0$}
			T_{r,k} &= \dfrac{T_{r,k-1}\cdot T_{r-1,k} + D}{T_{r-1,k-1}}
		\end{align*}
	where $D=cd-d_1d_2$.
\end{Prop}

\begin{proof}
	Since $T$ is a Generalized Rascal Triangle, $T_{r,k} = c+kd_1+rd_2+rkd$; thus, when $r,k\ge 1$
		\begin{align*}
			&T_{r,k-1} + T_{r-1,k} + d - T_{r-1,k-1} = (c+(k-1)d_1+rd_2+r(k-1)d)\\ 
			&\phantom{=} + (c+kd_1+(r-1)d_2+(r-1)kd) + d\\ 
			&\phantom{=}\ - (c+(k-1)d_1+(r-1)d_2+(r-1)(k-1)d)\\
			&= c+ kd_1 + rd_2 + r(k-1)d + (r-1)kd + d - (r-1)(k-1)d\\
			&= c+ kd_1 + rd_2 + rkd - rd + rkd - kd + d -rkd + rd + kd - d\\
			&= c+kd_1+rd_2+rkd\\ 
			&= T_{r,k}.\\ \\ \\
			\intertext{To show that $T$ has a Rascal-like multiplication rule, let $D=cd-d_1d_2$ and suppose $r,k\ge 1$. Then by \lemref{(linE)(linW)=(linS)(linN)}}
			&T_{r-1,k}\cdot T_{r,k-1} + D = \Big(c+(k-1)d_1+rd_2+r(k-1)d\Big)\Big (c+kd_1+(r-1)d_2+(r-1)kd\Big)\\ 
			&\phantom{=} + cd-d_1d_2\\
			&=  \Big(c+kd_1+rd_2+rkd\Big)\Big(c + (k-1)d_1 + (r-1)d_2 + (r-1)(k-1)d\Big)\\
			&=T_{r,k}T_{r-1,k-1}\\
			\intertext{Thus, when $T_{r-1,k-1}\ne 0$}
			&T_{r,k} = \dfrac{T_{r-1,k}\cdot T_{r,k-1} + D}{T_{r-1,k-1}}.
		\end{align*}
\end{proof}

Note that the additive constant $d$ for the Rascal-like addition rule is the same as the $d$ in the definition of the Generalized Rascal Triangle.
\medskip

\subsection{Rascal-like addition and multiplication imply a Generalized Rascal Triangle}\label{Ss:MultmakesGRT}

We now prove that a number triangle that has arithmetic sequences on the outside diagonals and satisfies either a Rascal-like addition rule or a Rascal-like multiplication rule for the interior numbers must be a Generalized Rascal Triangle.  

\begin{Prop}\label{add=grt}
Let $d_1, d_2, d\in \bbZ$ and $T$ be a number triangle with $T_{r,0} = T_{0,0} + rd_2$, $T_{0,k} = T_{0,0} + kd_1$, and  $T_{r,k} = T_{r,k-1}+ T_{r-1,k} + d - T_{r-1,k-1}$. Then there exists a constant $c\in \bbZ$ such that $T=\grt$.
\end{Prop}

\begin{proof}
	Let $c=T_{0,0}$.  To prove that $T=\grt$, we first note that 
	\begin{align*}
		T_{r,0} &= c + rd_2 = c+ 0d_1 + rd_2 + r\cdot 0 d\\
		\intertext{and}
		T_{0,k} &= c + kd_1 = c + kd_1 + 0d_2 + 0\cdot kd,\\ 
	\end{align*}
so on the exterior diagonals, $T_{r,k} = c + kd_1+rd_2 +rkd$.  For the interior numbers $T_{r,k}$ with $r,k\ge 1$, we prove that $T_{r,k} = c kd_1 + rd_2 +rkd$ by induction on the row number $N$ for $N\ge 2$. By \lemref{L:constantr+k}, $N=r+k$ for each entry $T_{r,k}$ on $N$.

Suppose $r, k \in \bbN$ with $r+k=2$.  Since $r,k \ge 1,\ r+k=2$ means $r=1$ and $k=1$ and so $T_{r,k-1} = T_{1,0} = c+ d_2$, $T_{r-1,k} = T_{0,1} = c+d_1$ and $T_{r-1,k-1}=T_{0,0}=c$.  Since
	\begin{align*}
		T_{1,1}&=T_{0,1}+T_{1,0} +d - T_{0,0}\\
		\intertext{we have}
		T_{1,1} &= c+d_1 + c+d_2 + d - c = c+d_1+d_2 + d.
	\end{align*}

Now suppose for that $r, k \in\bbN$ with $2\le r+k < N$, $T_{r,k} = c+kd_1+rd_2+rkd$. Then for $r+k=N$ we have $(r-1) + k = N-1$ and $r+(k-1) = N-1$. Using the induction hypothesis and the addition rule we get that
	\begin{align*}
		T_{r,k} &= T_{r,k-1}+ T_{r-1,k} + d - T_{r-1,k-1}\\
			&= \big(c + (k-1)d_1 + rd_2 + r(k-1)d\big) + \big(c + kd_1 + (r-1)d_2 + (r-1)kd\big) + d\\ 
			&\phantom{=} - \big(c + (k-1)d_1 + (r-1)d_2 + (r-1)(k-1)d\big)\\
			&=c +kd_1+rd_2 + rkd
	\end{align*}
Thus, $T_{r,k} = c + kd_1+rd_2 +rkd$ for $r,k\ge 0$, and so $T$ is a Generalized Rascal Triangle.
\end{proof}

\begin{Prop}\label{mult=grt}
Let $D, d_1,d_2 \in \bbZ$ and $T$ be a number triangle with $T_{r,0} = T_{0,0} + rd_2$, $T_{0,k} = T_{0,0} + kd_1$, $T_{r,k} \ne 0$ for all $r,k \ge 0$ and  $T_{r,k} = \dfrac{T_{r,k-1}\cdot T_{r-1,k} + D}{T_{r-1,k-1}}$. Then there exist constants $c,d\in \bbZ$ such that $T=\grt$ and $D=cd-d_1d_2$.
\end{Prop}

Note that since we are assuming all the interior numbers satisfy a Rascal-like multiplication rule, we require that $T_{r,k}\ne 0$ for all $r,k\ge 0$.

\begin{proof} We first establish the relation $D=cd-d_1d_2$. Let $c=T_{0,0}$, and $d=T_{1,1}-T_{0,1}-T_{1,0} + T_{0,0}$.  Since $T_{r,k}\ne 0$ for all, $r,k \ge 0$, we have that $c = T_{0,0}\ne 0$ and so
		\begin{align}
			d &= \dfrac{dT_{0,0}}{T_{0,0}}\notag\\
					&= \dfrac{T_{0,0}[(T_{1,1}-T_{0,1}-T_{1,0}+T_{0,0})]}{T_{0,0}}\notag\\
					&= \dfrac{T_{0,0}\cdot T_{1,1} - T_{0,0}\cdot T_{0,1} - T_{0,0}\cdot T_{1,0} + T_{0,0}^2}{T_{0,0}}\notag\\
					&= \dfrac{T_{0,1}\cdot T_{1,0} + D - T_{0,0}\cdot T_{0,1} - T_{0,0}\cdot T_{1,0} + T_{0,0}^2}{T_{0,0}}\notag\\
					&= \dfrac{D + (T_{0,1}\cdot T_{1,0} - T_{0,0}\cdot T_{0,1} - T_{0,0}\cdot T_{1,0} + T_{0,0}^2)}{T_{0,0}}\notag\\
					&= \dfrac{D+ (T_{0,1}-T_{0,0})(T_{1,0}-T_{0,0})}{T_{0,0}}\notag\\
					&= \dfrac{D+ d_1d_2}{c}\notag\\
			\intertext{Thus,}
			D &= cd - d_1d_2.
		\end{align}

To prove that $T=\grt$, we first note that 
	\begin{align*}
		T_{r,0} &= c + rd_2 = c+ 0d_1 + rd_2 + r\cdot 0 d\\
		\intertext{and}
		T_{0,k} &= c + kd_1 = c + kd_1 + 0d_2 + 0\cdot kd,\\ 
	\end{align*}
so on the exterior diagonals, $T_{r,k} = c + kd_1+rd_2 +rkd$.  For the interior numbers $T_{r,k}$ with $r,k\ge 1$, we prove $T_{r,k} = c kd_1 + rd_2 +rkd$ by induction on the row number $N$ for $N\ge 2$. By \lemref{L:constantr+k}, $N=r+k$ for each entry $T_{r,k}$ on $N$.

Suppose $r, k \in \bbN$ with $r+k=2$.  Since $r,k \ge 1,\ r+k=2$ means $r=1$ and $k=1$ and so $T_{r,k-1} = T_{1,0} = c+ d_2$, $T_{r-1,k} = T_{0,1} = c+d_1$ and $T_{r-1,k-1}=T_{0,0}=c$.  Since
	\begin{align*}
		d &= T_{1,1}-T_{0,1}-T_{1,0} + T_{0,0}\\
		\intertext{we have}
		T_{1,1} &= T_{0,1} + T_{1,0}-T_{0,0} + d= c + d_1 + d_2 + d.
	\end{align*}

Now suppose for that $r, k \in\bbN$ with $2\le r+k < N$, $T_{r,k} = c+kd_1+rd_2+rkd$. Then for $r+k=N$ we have $(r-1) + k = N-1$ and $r+(k-1) = N-1$. Using the induction hypothesis, the multiplication rule and  \lemref{(linE)(linW)=(linS)(linN)} we have that
	\begin{align*}
		T_{r,k} &= \dfrac{T_{r,k-1}\cdot T_{r-1,k} + D}{T_{r-1,k-1}}\\
		&= \dfrac{(c + (k-1)d_1 + rd_2 + r(k-1)d)(c + kd_1 + (r-1)d_2 + (r-1)kd) +cd-d_1d_2}{(c + (k-1)d_1 + (r-1)d_2 + (r-1)(k-1)d)}\\
		&= \dfrac{(c+kd_1+rd_2+rkd)(c + (k-1)d_1 + (r-1)d_2 + (r-1)(k-1)d)}{(c + (k-1)d_1 + (r-1)d_2 + (r-1)(k-1)d)}\\
		&= c+kd_1+rd_2+rkd\label{eq:Trkeq}
	\end{align*}
Thus, $T_{r,k} = c + kd_1+rd_2 +rkd$ for $r,k\ge 0$, and so $T$ is a Generalized Rascal Triangle.
\end{proof}

Combining \threepropref{prop:GRTAdd}{add=grt}{mult=grt} gives us the following theorem:

\begin{Thm}\label{thm:GRTMultThm}
	Let $c, d, d_1, d_2 \in \bbZ$ and $T$ be a number triangle with the arithmetic sequences $T_{0,k} = c +kd_1$ and $T_{r,0} = c+rd_2$ on the exterior diagonals. Then $T$ is the Generalized Rascal Triangle $T(c, d, d_1, d_2)$ and if and only if 
		\begin{align*}
			T_{r,k} &= c + kd_1+rd_2 +rkd\\
			\intertext{and whenever $T_{r-1,k-1}\ne 0$}
			T_{r,k} &= \dfrac{T_{r-1,k}\cdot T_{r,k-1} + D}{ T_{r-1,k-1}}.\\
		\end{align*}
	where $D=cd-d_1d_2$.
\end{Thm}

We should note that Generalized Rascal Triangles can contain entries that are zero, but because they have both a Rascal-like addition and a Rascal-like multiplication; we will, when necessary, use the Rascal-like addition instead of the Rascal-like multiplication.

\section{Properties of Generalized Rascal Triangles}\label{Ss:grtpatterns}

\subsection{Arithmetic diagonals implies a Generalized Rascal Triangle} As we observed in \twoeqref{eq:majorarthseq}{eq:minorarthseq}, if $c,d,d_1,d_2 \in \bbZ$, then one consequence of the definition of the Generalized Rascal Triangle $T(c,d,d_1,d_2)$ is that all major and minor diagonals are arithmetic sequences.  Not surprisingly, the converse of this observation is true; if all the major and minor diagonals of a number triangle, $T$, are arithmetic sequences, then $T$ is a Generalized Rascal Triangle.  

We start by showing that the constant differences for the arithmetic sequences on the diagonals change by a fixed amount as we move from one diagonal to the next.

\begin{Lma}\label{Lma:constdifffixedchange}
	Let $T$ be a number triangle with arithmetic sequences on all major and minor diagonals and let $M_r(k) = T_{r,0} + kc_r$ and $m_k(r) = T_{0,k}+rb_k$ denote the arithmetic sequences on the major and minor diagonals respectively.   Then there exists a constant $d\in \bbZ$ such that $d=c_r-c_{r-1} = b_k-b_{k-1}, \forall r,k \ge 1$. 
\end{Lma}  

\begin{proof}
	 Since $T_{r,k} = M_r(k)=m_k(r)$, $T_{r,k} = T_{r,k-1} + c_r = T_{r-1,k} + b_k\ \forall r,k\ge 1$. Let $r, k \ge 1$ be arbitrary; then, 
		\begin{align}
			T_{r,k} &= T_{r,k-1} + c_r = T_{r-1,k-1} + b_{k-1} + c_r\\
				     &= T_{r-1,k} + b_k  = T_{r-1,k-1} + c_{r-1} + b_k\\
			\intertext{which means}
			b_{k-1}+ c_r &= c_{r-1} + b_k\notag\\
			c_r-c_{r-1} &= b_k-b_{k-1}
		\end{align}
	Since this is true for all $r,k \ge 1$, if we let $d=c_1-c_0 = b_1-b_0$ we get that $d = c_r-c_{r-1}=b_k-b_{k-1}$ for all $r, k \ge 1$.
\end{proof}

\begin{Prop}
	Let $T$ be a number triangle with arithmetic sequences on all major and minor diagonals.  Then $T$ is a Generalized Rascal Triangle.
\end{Prop}
\begin{proof}
 We must show that there exists constants $c,d, d_1,d_2\in \bbZ$ such that $T=T(c, d, d_1,d_2)$.  Let $c=T_{0,0}$; and let $M_r(k) = T_{r,0} + kc_r$ and $m_k(r) = T_{0,k}+rb_k$ denote the arithmetic sequences on the major and minor diagonals respectively.  By \lemref{Lma:constdifffixedchange} there is a constant $d\in \bbZ$ so that $d=c_r-c_{r-1} = b_k-b_{k-1}, \forall r,k \ge 1$. Let
 	\begin{align*}
		d_1 &= T_{1,0} - T_{0,0} = c_0\\
		d_2 &= T_{0,1} - T_{0,0} = b_0\\
		\intertext{then}
		T_{r,k} &= M_r(k) = T_{r,0} + kc_r = m_0(r) + kc_r = T_{0,0} + rb_0 + kc_r\\
			     &= c + rd_2 + k(c_{r-1} + d) = \cdots = c + rd_2 + k(c_0 + rd)\\ 
			     &= c + rd_2 + kd_1 + rkd 
	\end{align*}
\end{proof}

\begin{Ex}
	Consider the number Triangle $W$ from \exref{ex4} in \secref{S:MExGRTs}:
		\begin{figure}[H]
			\centering
			\includegraphics[scale=0.5]{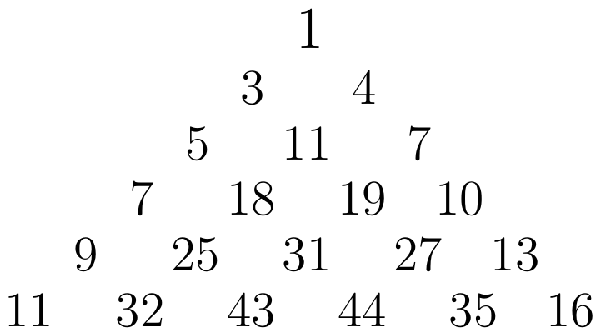}
			\caption{The first six rows of $W$.}\label{Fi:GenRasTri1}
		\end{figure}
	\noindent This number triangle has arithmetic sequences on all diagonals with $c=1$; the major diagonals are the arithmetic sequences 
		\begin{align*}
			1, 3, 5, 7, 9,&\dots\\
			4, 11, 18, 25, &\dots\\
			7, 19, 31, 43, &\dots\\
			\phantom{7,19,}\vdots \phantom{, 43,} &\phantom{\dots}\\
		\end {align*}
	\noindent so $d_1=2$.  The minor diagonals are the arithmetic sequences 
		\begin{align*}
			1, 4, 7, 10,13, &\dots\\
			3, 11, 19, 27, &\dots\\
			5, 18, 31, 44, &\dots\\
			\phantom{5,18,}\vdots \phantom{, 44,} &\phantom{\dots}\\
		\end{align*}
	\noindent which means $d_2 = 3$.  The differences for the arithmetic sequences change by 5 each time so $d=5$; and so 
		\begin{align*}
			T_{r,k} &= 1+3k+2r+5rk\\
			\intertext{which means $T=T(1,2,3,5)$.  The Rascal-like addition rule is}
			T_{r,k} &= T_{r-1,k} + T_{r,k-1} + 5 - T_{r-1,k-1}\\
			\intertext{For the Rascal-like multiplication rule}
			D&=cd-d_1d_2 = 1\cdot 5 - 2\cdot 3 = -1
			\intertext{which means}
			T_{r,k} &= \dfrac{T_{r-1,k}\cdot T_{r,k-1} - 1}{T_{r-1,k-1}}.
		\end{align*}
\end{Ex}
\subsection{Uniqueness of the Rascal Triangle} While we have seen that the Rascal Triangle is not the only number triangle that is generated by both a Rascal-like multiplication rule and an Rascal-like addition rule, the Rascal Triangle is unique in the sense that if $T$ is a Generalized Rascal Triangle with $d=1$ and $c,d_1,d_2\in \bbZ$ with $D=cd-d_1d_2 = c-d_1d_2=1$ then $T$ sits inside the Rascal Triangle as a sub-triangle.
	
\begin{Def}
	A number triangle $T'$ is called a {\bf sub-triangle} of a number triangle $T$ if there exists $r_0,k_0 \in \bbN$ such that $T'_{r,k} = T_{r_0+r,k_0+k}$.
\end{Def}

\begin{figure}[H]
	\centering
	\includegraphics[scale=0.375]{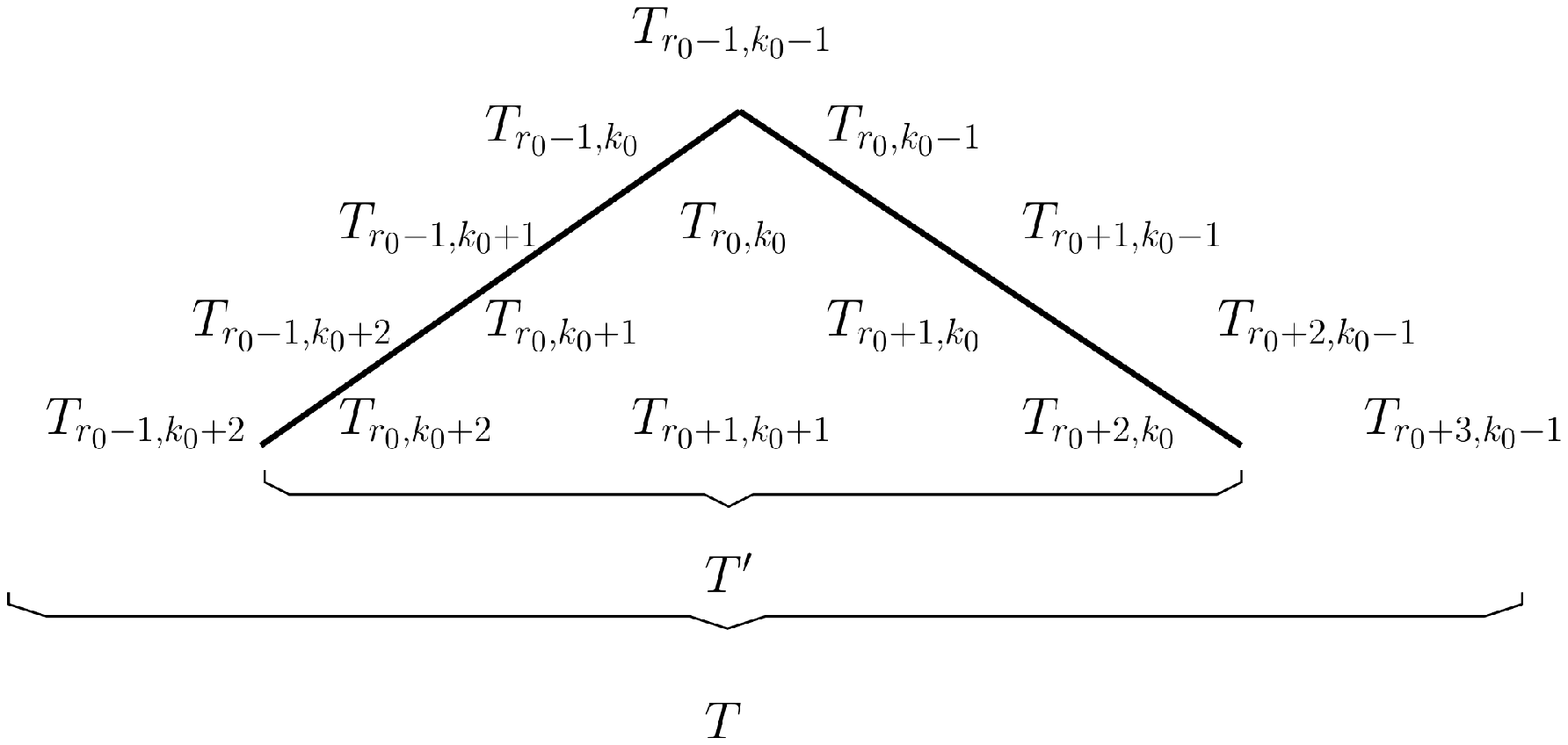}
	\caption{Sub-Triangle $T'$ starting at $T_{r_0,k_0}$.}
\end{figure}

\begin{Cor}\label{D=d=1}
	Let $c,d_1,d_2 \in \bbZ$, such that $c-d_1d_2=1$; and let $T$ be the Generalized Rascal Triangle $T(c,1,d_1,d_2)$.  Then $T$ is a sub-triangle of the Rascal Triangle, $R$.
\end{Cor}

\begin{proof}
	Since $c-d_1d_2=1$, $c=1+d_1d_2$, so $T_{0,0} = c=1+d_1d_2 = R_{d_1,d_2}$, and 
		\begin{align*}
			T_{r,k} &= c+kd_1 + rd_2 + rk\\
				&= 1+d_1d_2+ kd_1 + rd_2 + rk\\
				&=1+(d_1+r)(d_2+k)\\
				&=R_{d_1+r,d_2+k}
		\end{align*}
	Thus $T$ is the sub-triangle of $R$ starting at $R_{d_1,d_2}$.
\end{proof}
	
The original Rascal Triangle, $R$, corresponds to the Generalized Rascal Triangle, $T(1,1,0,0)$.  If we take $c=d$ and $d_1=d_2=0$, then $T(c,c,0,0)$ is a multiple of $R$.
	
\begin{Def} 
	A number triangle $T'$ is called a {\bf multiple} of a number triangle $T$ if there exists a constant $m$ such that for $r,k \ge 0$, $T'_{r,k} = mT_{r,k}$.
\end{Def}
	
\begin{Cor}\label{Cor:GRT=cRT}
	Let $c\in \bbZ$, and let $T$ be the Generalized Rascal Triangle, $T(c,c,0,0)$.  Then $T=cR$.
\end{Cor}

The proof of this is left to the reader.
\medskip
	
\subsection{Student Discovered Properties in Generalized Rascal Triangles.} In the Fall 2015, I challenged students in one of my Mathematics for Liberal Arts classes to find patterns in the Rascal Triangle.  To my delight, they discovered several properties that, as far as I can tell, were unknown at the time; see \textit{Student Inquiry and the Rascal Triangle} \cite{H} for details about my student's work.   Further investigations showed that these properties were also present in Generalized Rascal Triangles;  we conclude by presenting proofs of several of these properties, as well as some others, for Generalized Rascal Triangles.   

\begin{Prop}
	Let $c,d,d_1,d_2\in\bbZ$ and $T=\grt$ be the associated Generalized Rascal Triangle; then the row sum, $s_n$, for the $\nth$ row is 
				\begin{equation*}
					s_n = \dfrac{d}6 n^3 + \left(\dfrac{d_1+d_2}2\right)n^2 + \left(c+\dfrac{d_1+d_2}2 - \dfrac{d}6\right)n +c.
				\end{equation*}
\end{Prop}

\begin{proof}
	Since $T$ is a Generalized Rascal Triangle,  $T_{r,k} = c+ kd_1 + rd_2 + rkd$, and by \lemref{L:constantr+k} we have $k=n-r$ for every entry, $T_{r,k}$, on the $\nth$ row. Thus
		\begin{align*}
			&s_n =\dsum_{r=0}^n (c+(n-r)d_1+rd_2 +r(n-r)d)\\
				&= \dsum_{r=0}^n c + \dsum_{r=0}^n d_1 - \dsum_{r=0}^n d_1r + \dsum_{r=0}^n  d_2r + \dsum_{r=0}^n n\,d\,r - \dsum_{r=0}^n d\,r^2\\
				&= c(n+1) +(n^2+n)d_1 - \left(\dfrac{n^2+n}2\right)d_1 +(n^2+n)d_2 + \left(\dfrac{n^3+n^2}2\right)d\\
				&\phantom{=} - \left(\dfrac{2n^3+3n^2 + n}6\right)d\\
				&= \dfrac{d}6 n^3 + \left(\dfrac{d_1+d_2}2\right)n^2 + \left(c+\dfrac{d_1+d_2}2 - \dfrac{d}6\right)n +c.
		\end{align*}
\end{proof}

For the next pattern, we need the following definition.

\begin{Def}
	Let $T$ denote a number triangle. For $n\ge 1$, an {\bf $\boldsymbol{n}$-diamond} in $T$ is the diamond whose sides are formed by the entries $T_{r,k}$ to $T_{r+n-1,k}$ on the $\kth$ minor diagonal, $T_{r+n,k}$ to $T_{r+n,k+n-1}$ on the $(r+n)^{\tth}$ major diagonal, $T_{r+n,k+n}$ to $T_{r+1,k+n}$ on the $(k+n)^{\tth}$ minor diagonal and $T_{r,k+n}$ to $T_{r,k+1}$ on the $\rth$ major diagonal.  See \figref{Fi:diamond}.  We call $T_{r,k}$ the top number of the diamond.
\end{Def}

\begin{figure}[H]
	\centering
	\includegraphics[scale=0.5]{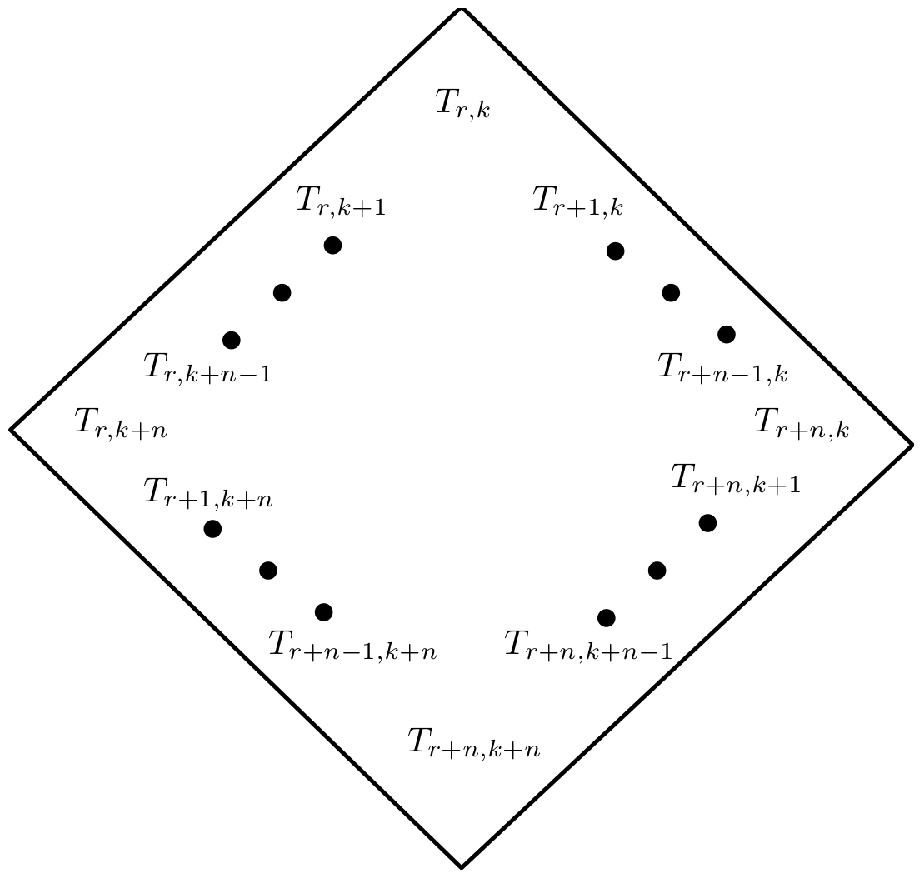}
	\caption{An $n$-diamond.}\label{Fi:diamond}
\end{figure}

The following property was named after John, who discovered the original pattern in the Rascal Triangle, $T(1,1,0,0)$.
		
\begin{Prop}[John's Odd/Even Diamond Patterns]\label{P:EvenOddPatterns}
	Let $c,d,d_1,d_2\in\bbZ$ and $T=\grt$ be the associated Generalized Rascal Triangle; then $T$ has the following diamond patterns: 
		\begin{enumerate}
			\item[i.] \textnormal{(Odd Diamond Pattern)} Let $D$ be a $(2n+1)$-diamond in $T$ whose top number is $T_{r,k}$.  The average of the $8n$ entries along the edge of the diamond is the number at the middle of the diamond, $T_{r+n,k+n}$.  That is,
				\begin{align}
					&\dfrac{\left(\dsum_{i=0}^{2n-1}T_{r+i,k} + \dsum_{i=0}^{2n-1} T_{r+2n,k+i} + \dsum_{i=0}^{2n-1}T_{r+2n-i,k+2n} + \dsum_{i=0}^{2n-1}T_{r,k+2n-i}\right)}{8n}\notag\\ 
					&\phantom{=}= T_{r+n,k+n}.
				\end{align}
			\item[ii.] \textnormal{(Even Diamond Pattern)} Let $D_1$ be a $2$-diamond whose top number is $T_{r,k}$, and for $n \le \min\{r,k\}$, let $D_n$ denote the $2n$-diamond whose top entry is $T_{r-n,k-n}$. Then the average of the entries along the edges of $D_n$ is equal to the average of the four entries along the edges of $D_1$.  That is,
				\begin{align}
					&\dfrac{\left(\dsum_{i=0}^{2n-2}T_{r-n+i,k-n} + \dsum_{i=0}^{2n-2} T_{r+n+1,k-n+i} + \dsum_{i=0}^{2n-2}T_{r+n+1-i,k+n+1} + \dsum_{i=0}^{2n-2}T_{r-n,k+n-i}\right)}{8n-4}\notag\\
					&= \dfrac{T_{r,k} + T_{r+1,k} + T_{r+n,k+n} +T_{r,k+1}}4.
				\end{align}
		\end{enumerate}
\end{Prop}

\begin{Ex}
	In the $3$-diamond in \figref{Fi:OddDiamondEx}, the average of the 8 numbers along the edge of the $3$-diamond is equal to $50$.
		\begin{figure}[H]
			\centering
			\includegraphics[scale=0.5]{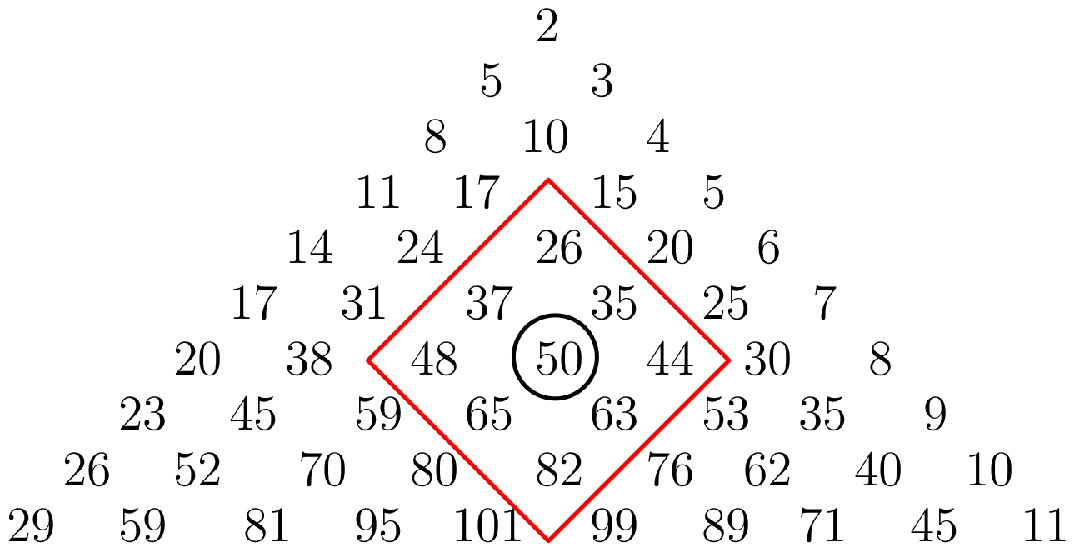}
			\caption{John's Odd Diamond Pattern}\label{Fi:OddDiamondEx}
		\end{figure}
		
	In the $4$-diamond in \figref{Fi:EvenDiamondEx}, the average of the $12$ numbers along the edge of the $4$-diamond is equal to the average of the $4$ numbers in the $2$-diamond in the center.
		\begin{figure}[H]
			\centering
			\includegraphics[scale=0.5]{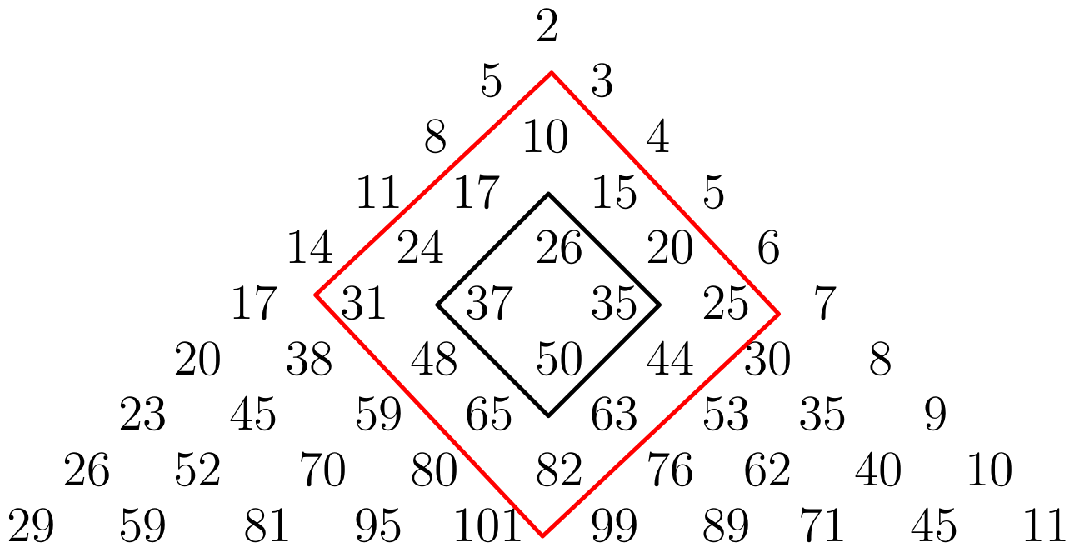}
			\caption{John's Even Diamond Pattern}\label{Fi:EvenDiamondEx}
		\end{figure}
		
\end{Ex}

\begin{proof}[Proof of \propref{P:EvenOddPatterns}]
	For the Odd Diamond Pattern, we can regroup the terms in the numerator as follows
		\begin{equation*}
			\dsum_{i=0}^{2n-1}\bigg[\Big(T_{r+i,k}+T_{r+2n-i,k+2n}\big) +\big(T_{r+2n,k+i} + T_{r,k+2n-i}\Big)\bigg]
		\end{equation*}
		Since $T$ is a Generalized Rascal Triangle we have that $T_{r,k} = c+ kd_1+rd_2+rkd$; so    
			\begin{align}
				&T_{r+i,k} + T_{r+2n-i,k+2n} = (c+kd_1 +(r+i)(d_2+kd)) + (c+ (k+2n)d_1\notag\\ 
				&\phantom{=} + (r+2n-i)(d_2+(k+2n)d))\notag\\
				&= c + kd_1 + (r+i)(d_2+kd) + c + kd_1 + 2nd_1 + (r-i)(d_2+kd)\notag\\ 
				&\phantom{=} + (r-i)2nd + 2n(d_2+kd) + 4n^2d\notag\\
				&= 2c + 2kd_1 + 2rd_2+rkd + id_2 + ikd + 2nd_1 + rkd -id_2 -ikd + 2nrd\notag\\ 
				&\phantom{=} -2ndi + 2nd_2 + 2nkd + 4n^2d\notag\\
				&= 2c + (2kd_1 +2nd_1) + (2rd_2 + 2nd_2) + (2rkd + 2rnd + 2knd + 2n^2d) + 2n^2d\notag\\ 
				&\phantom{=} - 2ndi\notag\\
				&= 2(c + (k+n)d_1 + (r+n)d_2 + (r+n)(k+n)d) + 2n^2d - 2ndi\notag\\
				&=2T_{r+n,k+n} + 2n^2d - 2ndi;\label{diamondsum1}\\
				\intertext{and}
				&T_{r+2n,k+i} + T_{r,k+2n-i}  = (c+(r+2n)d_2 + (k+i)(d_1 +(r+2n)d))\notag\\
				&\phantom{=} + (c+rd_2 + (k+2n-i)(d_1 + rd))\notag\\
				&= c+ rd_2+2nd_2 + (k+i)(d_1+rd) + (k+i)2nd + c + rd_2 + (k-i)(d_1+rd)\notag\\ 
				&\phantom{=} + 2n(d_1+rd)\notag\\
				&= c + rd_2 + 2nd_2 + kd_1+rkd + id_1 + ird + 2knd + 2ndi + c + kd_1 + rd_2+rkd\notag\\ 
				&\phantom{=} - id_1 -ird + 2nd_1 + 2nrd + 2n^2d - 2n^2d\notag\\
				&= 2c+ (2rd_2+2nd_2) + (2kd_1+2nd_1) + (2rkd + 2rnd + 2knd + 2n^2d) + 2ndi\notag\\ 
				&\phantom{=} - 2n^2d\notag\\
				&= 2(c+ (k+n)d_1 + (r+n)d_2 + (r+n)(k+n)d) + 2ndi - 2n^2d\notag\\
				&=2T_{r+n,k+n} + 2dni - 2n^2d\label{diamondsum2}\\ 
				\intertext{Combining \twoeqref{diamondsum1}{diamondsum2} we get}
				&(T_{r+i,k} + T_{r+2n-i,k+2n}) + (T_{r+2n,k+i}+ T_{r,k+2n-i})\notag\\ 
				&= (2T_{r+n,k+n} + 2n^2d - 2ndi ) + (2T_{r+n,k+n} + 2ndi - 2n^2d)\notag\\ 
				&=4T_{r+n,k+n}
			\end{align}
		Therefore
			\begin{align}
				&\dfrac{\left(\dsum_{i=0}^{2n-1}T_{r+i,k} + \dsum_{i=0}^{2n-1} T_{r+2n,k+i} + \dsum_{i=0}^{2n-1}T_{r+2n-i,k+2n} + \dsum_{i=0}^{2n-1}T_{r,k+2n-i}\right)}{8n}\notag\\ 
				&=\dfrac{\dsum_{i=0}^{2n-1}\Big[(T_{r+i,k}+T_{r+2n-i,k+2n}) + (T_{r+2n,k+i} + T_{r,k+2n-i})\Big]}{8n}\notag\\
				&=\dfrac{\dsum_{i=0}^{2n-1}4T_{r+n,k+n}}{8n}=\dfrac{8n(T_{r+n,k+n})}{8n}= T_{r+n,k+n}.
			\end{align}

		For the Even Diamond Pattern, we can regroup the terms in the sum of the edges of a $2n$-diamond as follows:
			\begin{equation*}
				\dsum_{i=0}^{2n-2}\Big(T_{r-n+i,k-n}+T_{r+n+1,k-n+i} + T_{r+n+1-i,k+n+1} + T_{r-n,k+n+1-i}\Big)
			\end{equation*}
	
		Since $T$ is a Generalized Rascal Triangle we have that $T_{r,k} = c+ kd_1+rd_2+rkd$.  Thus,  
			\begin{align}
				T_{r-n+i,k-n} &= c + (k-n)d_1 + (r-n+i)d_2 + (r-n+i)(k-n)d\notag\\
							&= c + (k-n)d_1 + (r-n)d_2 + id_2 + (r-n)(k-n)d + i(k-n)d\notag\\
							&= T_{r-n,k-n} + id_2 + i(k-n)d\label{diamondedgesum1}\\
				T_{r+n+1,k-n+i} &= c + (k-n+i)d_1 + (r+n+1)d_2 + (r+n+1)(k-n+i)d\notag\\
							&= c + (k-n)d_1 + id_1 + (r+n+1)d_2 + (r+n+1)(k-n)d\notag\\ 
							&\phantom{=} + i(r+n)d+id\notag\\
							&= T_{r+n+1,k-n} + id_1 + i(r+n)d + id\label{diamondedgesum2}\\
				T_{r+n+1-i,k+n+1} &= c + (k+n+1)d_1 + (r-n+1-i)d_2 + (r+n+1-i)(k+n+1)d\notag\\
							&= c + (k+n+1)d_1 + (r+n+1)d_2 - id_2 + (r+n+1)(k+n+1)d\notag\\ 
							&\phantom{=} - i(k+n)d - id\notag\\
							&= T_{r+n+1,k+n+1} - id_2 -  i(k+n)d - id\label{diamondedgesum3}\\
				T_{r-n,k-n-i} &= c + (k+n+1-i)d_1 + (r-n)d_2 + (r-n)(k+n+1-i)d\notag\\
							&= c + (k+n+1)d_1 - id_1 + (r-n)d_2 + (r-n)(k+n+1)d - i(r-n)d\notag\\
							&= T_{r-n,k+n+1} - id_1 - i(r-n)d\label{diamondedgesum4}
			\end{align}
	
		Combining \conseceqref{diamondedgesum1}{diamondedgesum4} we get
			\begin{align}
				T_{r-n+i,k-n} &+ T_{r+n+1,k-n+i} + T_{r+n+1-i,k+n+1} + T_{r-n,k+n+1-i}\notag\\ 
				&= T_{r-n,k-n} + id_2 + i(k-n)d + T_{r+n+1,k-n} + id_1 + i(r+n)d + id\notag\\ 
				&\phantom{=} + T_{r+n+1,k+n+1} - id_2 -  i(k+n)d - id+ T_{r-n,k+n+1} - id_1 - i(r-n)d\notag\\ 
				&=T_{r-n,k-n} + T_{r+n+1,k-n} + T_{r+n+1,k+n+1} + T_{r-n+i,k+n-1} 
			\end{align}
		Therefore
			\begin{align*}
				&\dfrac{\dsum_{i=0}^{2n-2}(T_{r-n+i,k-n} +T_{r+n+1,k-n+i} + T_{r+n+1-i,k+n+1} + T_{r-n,k+n+1-i})}{8n-4}\\ 
				&= \dfrac{\dsum_{i=0}^{2n-2}(T_{r-n,k-n} + T_{r+n+1,k-n} + T_{r+n+1,k+n+1} + T_{r-n+i,k+n-1})}{8n-4}\\ 
				&= \dfrac{(2n-1)(T_{r-n,k-n} + T_{r+n+1,k-n} + T_{r+n+1,k+n+1} + T_{r-n+i,k+n-1})}{8n-4}\\
				&=  \dfrac{(T_{r-n,k-n} + T_{r+n+1,k-n} + T_{r+n+1,k+n+1} + T_{r-n+i,k+n-1})}4.
			\end{align*}
\end{proof}

The following properties were named after the student Ashley, who discovered the original version for the Rascal Triangle, $T(1,1,0,0)$.

\begin{Prop}[Ashley's Rule]\label{P:Ashley}
	Let $c,d,d_1,d_2\in\bbZ$ and $T=\grt$ be the associated Generalized Rascal Triangle; then 
			\begin{equation*}
				T_{r,k} = T_{r-1,k}+T_{r,k-1} - T_{r-2,k-1} + ((2-k)d - d_2).
			\end{equation*}
		for $r\ge 2, k\ge 1$
\end{Prop}

\begin{Ex}
	In \figref{Fi:Ashley1} 
		$$T_{r-1,k}+T_{r,k-1} - T_{r-2,k-1} + ((2-k)d - d_2) = 82 + 76 - 50 -9 = 99 = T_{r,k}.$$
		\begin{figure}[H]
			\centering
			\includegraphics[scale=0.5]{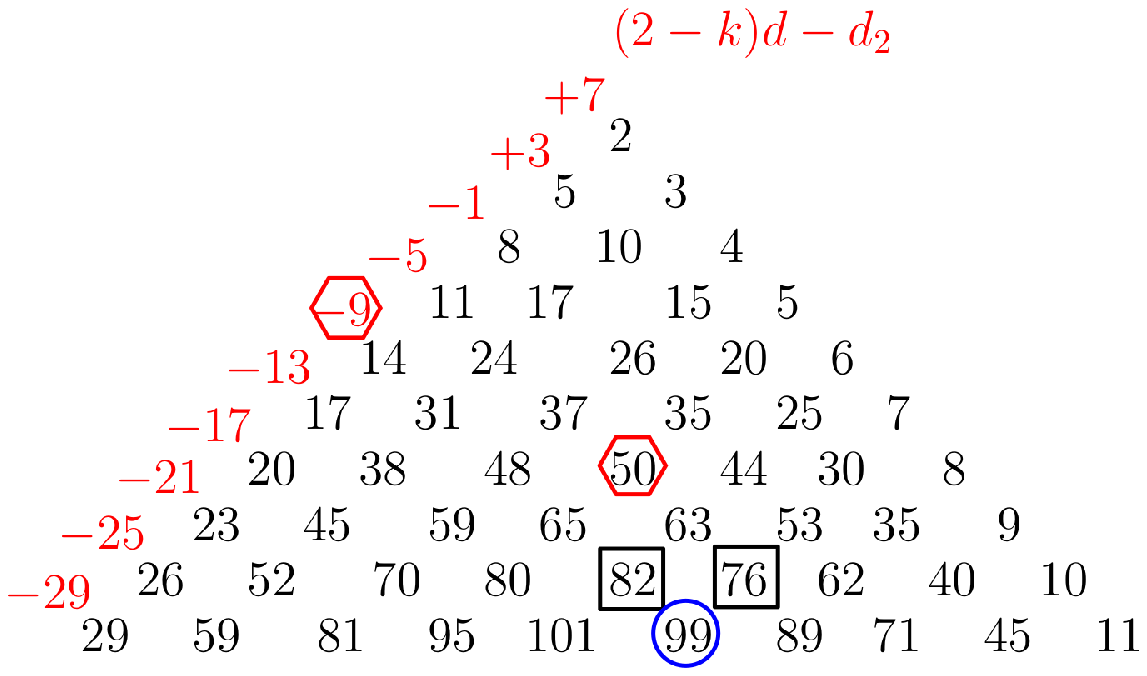}
			\caption{Ashley's Rule}\label{Fi:Ashley1}
		\end{figure}	
\end{Ex}

\begin{proof}[Proof of \propref{P:Ashley}]
	Recall that on the $\kth$ minor diagonal, the entries form the arithmetic progression $c + kd_1 + r(kd + d_2)$ with common difference $kd+d_2$, hence $T_{r-1,k} + kd + d_2 = T_{r,k}$ for $r\ge 1,k \ge 0$. Thus
		\begin{align*}
			&T_{r,k-1}+T_{r-1,k} - T_{r-2,k-1} + ((2-k)d - d_2)\\ 
			&= T_{r,k-1}+T_{r-1,k} - T_{r-2,k-1} - (k-1)d - d_2 + d \\
			&= T_{r,k-1}+T_{r-1,k} - (T_{r-2,k-1} + (k-1)d + d_2) + d\\
			&= T_{r,k-1}+T_{r-1,k} - T_{r-1,k-1} + d\\
			&= T_{r,k}
		\end{align*}
\end{proof}

My colleague, Julian Fleron and I subsequently discovered three ways of modifying Ashley's Rule so that the "diagonal factor", $(2-k)d - d_2$ was not needed.

\begin{Prop}[Modified Ashley Rules]\label{P:ModAshley}
	Let $c,d,d_1,d_2\in\bbZ$ and $T=\grt$ be the associated Generalized Rascal Triangle; then 
		\begin{align*}
			T_{r,k} &= T_{r-1,k} + T_{r,k-1} - T_{r-2,k-1} - T_{r-2,k-2}  +T_{r-3,k-2}\\
			\intertext{for $r\ge 3, k\ge 2$;}
			T_{r,k} &= T_{r,k-1} + T_{r-1,k-1} - T_{r-2,k-2} - T_{r-2,k-3}  +T_{r-3,k-3}\\
			\intertext{for $r, k\ge 3$; and}
			T_{r,k} &= T_{r-1,k} + T_{r-1,k-1} - T_{r-2,k-2} - T_{r-3,k-2}  +T_{r-3,k-3}\\
		\end{align*}
		for $r, k\ge 3$.
\end{Prop}

\begin{Ex}
	In \figref{Fi:Ashley2} for the first modification, $T_{r,k} = T_{r-1,k} + T_{r,k-1} - T_{r-2,k-1} - T_{r-2,k-2}  +T_{r-3,k-2}$,
		$$T_{r-1,k}+T_{r,k-1} - T_{r-2,k-1}  - T_{r-2,k-2}  +T_{r-3,k-2} = 82 + 76 - 50 - 35 + 26 = 99 = T_{r,k}.$$
		\begin{figure}[H]
			\centering
			\includegraphics[scale=0.5]{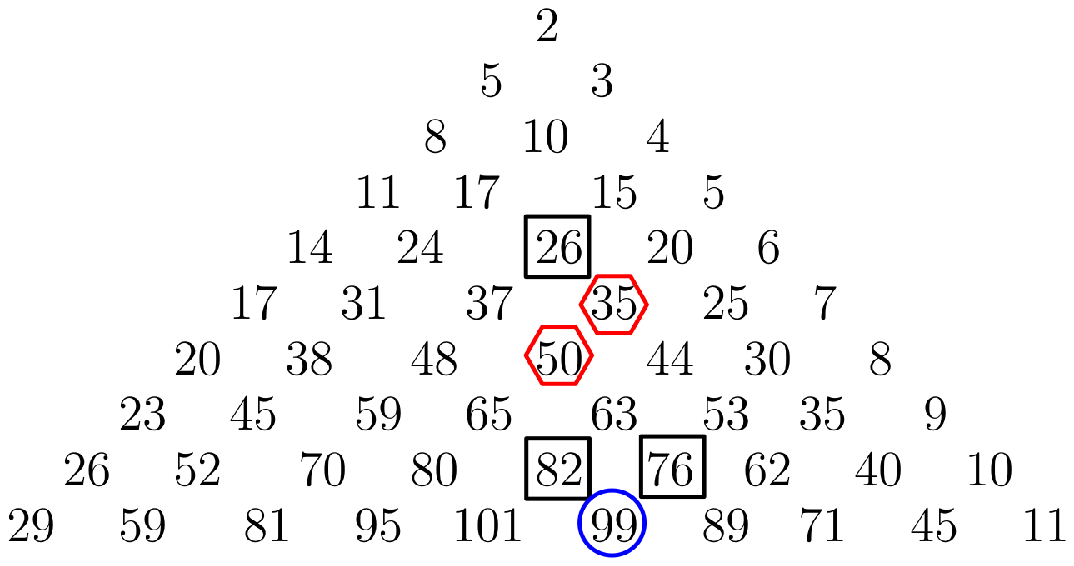}
			\caption{Modification 1 of Ashley's Rule}\label{Fi:Ashley2}
		\end{figure}	

	In \figref{Fi:Ashley3} for the second modification, $T_{r,k} = T_{r,k-1} + T_{r-1,k-1} - T_{r-2,k-2} - T_{r-2,k-3}  +T_{r-3,k-3}$,
		$$T_{r,k-1} + T_{r-1,k-1} - T_{r-2,k-2} - T_{r-2,k-3}  +T_{r-3,k-3} = 76 + 63 - 35 - 20 + 15 = 99 = T_{r,k}.$$
		\begin{figure}[H]
			\centering
			\includegraphics[scale=0.5]{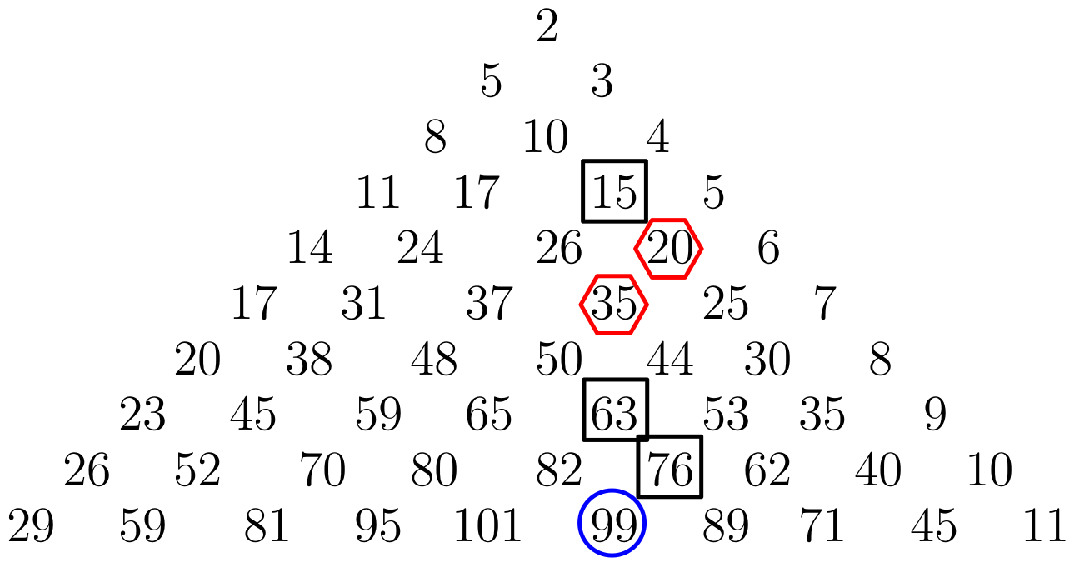}
			\caption{Modification 2 of Ashley's Rule}\label{Fi:Ashley3}
		\end{figure}	

	In \figref{Fi:Ashley4} for the third modification, $T_{r,k} = T_{r-1,k} + T_{r-1,k-1} - T_{r-2,k-2} - T_{r-3,k-2}  +T_{r-3,k-3}$,
		$$T_{r-1,k} + T_{r-1,k-1} - T_{r-2,k-2} - T_{r-3,k-2}  +T_{r-3,k-3} = 82 + 63 - 35 - 26 + 15 = 99 = T_{r,k}.$$
		\begin{figure}[H]
			\centering
			\includegraphics[scale=0.5]{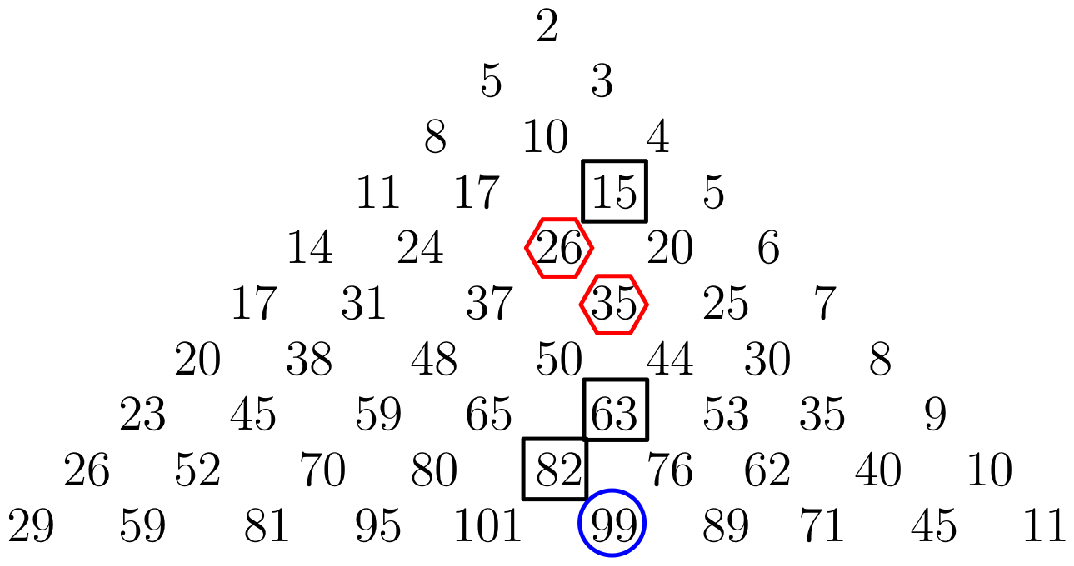}
			\caption{Modification 3 of Ashley's Rule}\label{Fi:Ashley4}
		\end{figure}	
\end{Ex}

\begin{proof}[Proof of \propref{P:ModAshley}] For Modification 1: 
	\begin{align*}
		T_{r-1,k} &+ T_{r,k-1} - T_{r-2,k-1} - T_{r-2,k-2}  + T_{r-3,k-2}\\ 
				&= (c+ (r-1)d_2+kd_1 + (r-1)kd)\\ 
				&\phantom{=}+ (c +(k-1)d_1 + rd_2 + r(k-1)d)\\  
				&\phantom{=} -  (c +(k-1)d_1 + (r-2)d_2+ (r-2)(k-1)d)\\ 
				&\phantom{=} -  (c+ (k-2)d_1+(r-2)d_2 + (r-2)(k-2)d)\\  
				&\phantom{=}+  (c+ (k-2)d_1+(r-3)d_2 + (r-3)(k-2)d)\\
				&=\phantom{+} c + kd_1 \phantom{- 2d_1} + rd_2- \phantom{2}d_2 + rkd - \phantom{2}kd\\
				&\phantom{=} + c + kd_1 - \phantom{2}d_1 + rd_2 \phantom{- 2d_2} + rkd  \phantom{-2kd} - \phantom{2}rd \\
				&\phantom{=} - c - kd_1 + \phantom{2}d_1 - rd_2 + 2d_2 - rkd  + 2kd + rd - 2d\\
				&\phantom{=}  - c - kd_1 + 2d_1 - rd_2 + 2d_2 - rkd + 2kd + 2rd - 4d\\ 
				&\phantom{=} + c + kd_1 - 2d_1 + rd_2 - 3d_2 + rkd - 3kd - 2rd + 6d\\
				&= c + kd_1 + rd_2 + rkd = T_{r,k}
				\intertext{For Modification 2:}
				T_{r,k-1} &+ T_{r-1,k-1} - T_{r-2,k-2} - T_{r-2,k-3}  + T_{r-3,k-3}\\ 
				&= (c + (k-1)d_1+ rd_2 + r(k-1)d)\\ 
				&\phantom{=}+ (c+(k-1)d_1 + (r-1)d_2 + (r-1)(k-1)d)\\ 
				&\phantom{=} -  (c + (k-2)d_1 + (r-2)d_2 + (r-2)(k-2)d)\\ 
				&\phantom{=} -  (c + (k-3)d_1 + (r-2)d_2 + (r-2)(k-3)d)\\  
				&\phantom{=} +  (c +(k-3)d_1 + (r-3)d_2 + (r-3)(k-3)d)\\
				&=\phantom{+} c + kd_1 - d_1 + rd_2\phantom{- 2d_2} + rkd - rd\\
				&\phantom{=} + c + kd_1 - d_1 + rd_2 - d_2 + rkd  - rd - kd + d\\
				&\phantom{=} - c - kd_1 + 2d_1 - rd_2 + 2d_2 - rkd  + 2rd + 2kd - 4d\\
				&\phantom{=}  - c - kd_1 + 3d_1 - rd_2 + 2d_2 - rkd + 3rd + 2kd - 6d\\ 
				&\phantom{=} + c + kd_1 - 3d_1 + rd_2 - 3d_2 + rkd - 3rd - 3kd + 9d\\
				&= c + kd_1 + rd_2 + rkd = T_{r,k}
		\end{align*}
	The proof for Modification 3 follows immediately from the proof of Modification 2 and \lemref{L:constantdiffs} by replacing $T_{r,k-1}$ by $T_{r-1,k} + d_2-d_1 + (k-r)d$ and $T_{r-2,k-3}$ by $T_{r-3,k-2} +  d_2-d_1 + (k-r)d$.
\end{proof}

\begin{Lma}[Constant Difference Along Adjacent Columns]\label{L:constantdiffs}
	Let  $c,d,d_1,d_2\in\bbZ$ and $T=\grt$ be the associated Generalized Rascal Triangle; then 
		$$T_{r,k}-T_{r-1,k+1} = T_{r-1,k-1} - T_{r-2,k}.$$
\end{Lma}

\begin{Ex}
	In \figref{Fi:ConstColDiff} 
		$$82-76 = 50-44 = 26-20 = 10-4 = 6.$$
		\begin{figure}[H]
			\centering
			\includegraphics[scale=0.5]{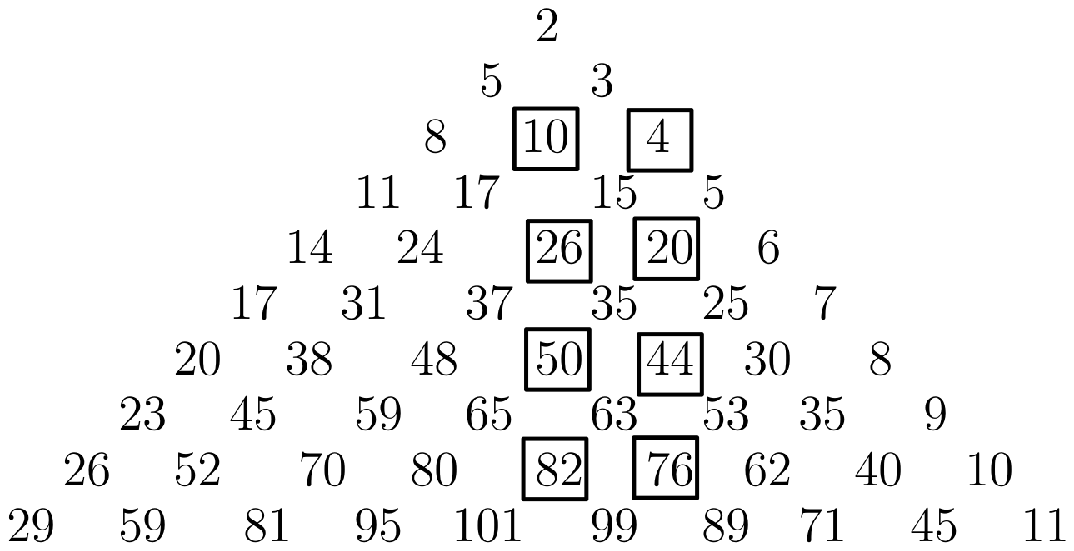}
			\caption{Constant Difference Along Columns}\label{Fi:ConstColDiff}
		\end{figure}	
\end{Ex}

\begin{proof}
	\begin{align*}
		T_{r,k} - T_{r-1,k+1} &= c + kd_1 + rd_2 +rkd\\ 
		&\phantom{=} - (c +(k+1)d_1 + (r-1)d_2 + (r-1)(k+1)d)\\
		&= d_2-d_1+(k-r+1)d\\
		T_{r-1,k-1} - T_{r-2,k} &= c +(k-1)d_1 + (r-1)d_2 + (r-1)(k-1)d\\ 
		&\phantom{=} -(c + kd_1 + (r-2)d_2 + (r-2)kd)\\
		&= d_2 -d_1 +(k-r+1)d
	\end{align*}
\end{proof}

The next property was named after Timothy and Meg who originally discovered the version for the Rascal Triangle, $T(1,1,0,0)$.   This property only applies to Generalized Rascal Triangles with $d_1=d_2 =0$.

\begin{Prop}[T-Meg Rule]\label{P:T-Meg} 
	Let $c,d\in\bbZ$ and $T=T(c,d,0,0)$ be the associated Generalized Rascal Triangle; then 
		$$T_{r,k}= T_{r-1,k-1} + T_{0,r+k-2} +T_{1,r+k-3}+ 2(d-c)$$ 
	for $r,\ge 1 k\ge 2$.
\end{Prop}

\begin{Ex}
	In \figref{Fi:T-Meg} $c=3, d=1$, and
		$$12 = T_{r,k} = T_{r-1,k-1} + T_{0,r+k-2} +T_{1,r+k-3}+ 2(d-c) = 7 + 3 + 6 - 4$$
		\begin{figure}[H]
			\centering
			\includegraphics[scale=0.5]{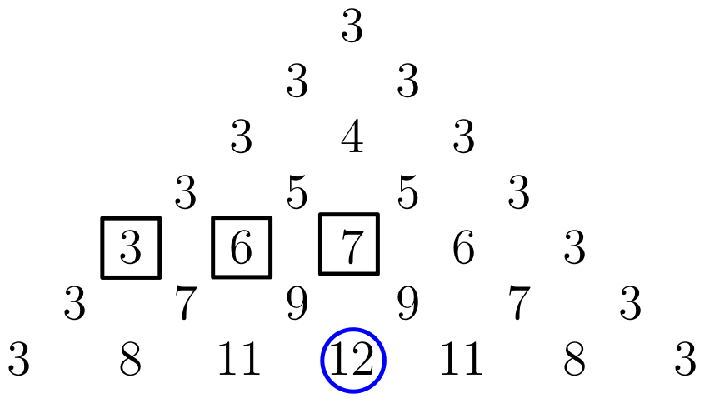}
			\caption{T-Meg Rule}\label{Fi:T-Meg}
		\end{figure}	
\end{Ex}

\begin{proof}[Proof of \propref{P:T-Meg}]
	Since $d_1=d_2=0$ we have $T_{r,k} = c+rkd$. Thus,  
	\begin{align*}
		T_{0,r+k-2} &+ T_{1,r+k-3}  +T_{r-1,k-1} + 2(d - c)\\
		 &= (c + 0(r+k-2)d) + (c+ 1(r+k-3)d) + (c+ (r-1)(k-1)d) + 2d - 2c\\
		&= c + c+rd+kd-3d + c+ rkd-rd-kd+d + 2d - 2c\\
		&= c+rkd = T_{r,k}
	\end{align*}
\end{proof}

\section{Conclusion}

The results in this paper grew out of explorations by Mathematics for Liberal Arts students looking for patterns in the Rascal Triangle.     My students enthusiasm and insights inspired me to look more deeply at the structure of the Rascal Triangle and the roles that \twoeqref{eq:RasForm1}{eq:RasForm2} played in that structure, which led to the Generalized Rascal Triangles.

\end{document}